\documentclass[11pt]{amsart}
\usepackage{amscd,amsmath,amsthm,amssymb,xypic}
\usepackage{hyperref}

\theoremstyle{plain}

\newtheorem{thm}{Theorem}[section]
\newtheorem{prop}[thm]{Proposition}
\newtheorem{lem}[thm]{Lemma}

\theoremstyle{definition}

\newtheorem{prop-defn}[thm]{Proposition--Definition}

\theoremstyle{remark}

\newtheorem{rem}[thm]{Remark}
\newtheorem{ex}[thm]{Example}
\newtheorem{con}[thm]{Construction}

\numberwithin{equation}{section}

\DeclareMathOperator{\Aut}{Aut}
\DeclareMathOperator{\Cl}{Cl}

\DeclareMathOperator{\Hom}{Hom}
\DeclareMathOperator{\cHom}{\mathcal{H}\mathnormal{om}}
\DeclareMathOperator{\Ext}{Ext}

\DeclareMathOperator{\Pic}{Pic}

\DeclareMathOperator{\End}{End}
\DeclareMathOperator{\GL}{GL}
\DeclareMathOperator{\rk}{rk}

\newcommand{\QED}{\ifhmode\unskip\nobreak\fi\quad {\rm Q.E.D.}} 

\newcommand{\bC}{\mathbb C}

\newcommand{\bP}{\mathbb P}
\newcommand{\bQ}{\mathbb Q}
\newcommand{\bR}{\mathbb R}

\newcommand{\bZ}{\mathbb Z}

\newcommand{\cB}{\mathcal B}
\newcommand{\cC}{\mathcal C}
\newcommand{\cX}{\mathcal X}
\newcommand{\cY}{\mathcal Y}

\newcommand{\cE}{\mathcal E}

\newcommand{\sH}{\mathcal H}

\newcommand{\cO}{\mathcal O}
\newcommand{\sL}{\mathcal L}

\newcommand{\ocM}{\overline{\mathcal M}}

\DeclareMathOperator{\cEnd}{\mathcal{E}\mathnormal{nd}}

\DeclareMathOperator{\spz}{sp}

\title{Exceptional bundles associated to degenerations of surfaces}
\date{\today}
\author{Paul Hacking}
\address{Department of Mathematics and Statistics, Lederle Graduate
Research Tower, University of Massachusetts, Amherst, MA 01003-9305}
\email{hacking@math.umass.edu}
\thanks{The author was partially supported by NSF grant DMS-0968824. I would like to thank T.~Bridgeland, I.~Dolgachev, D.~Huybrechts, A.~King, J.~Koll\'ar, J.~Tevelev, R.~Thomas, and G.~Urz\'ua for helpful discussions and correspondence.}
\begin{document}
\maketitle



\section{Introduction}

In 1981 J. Wahl described smoothings of surface quotient singularities with no vanishing cycles \cite[5.9.1]{W81}. Given a smoothing of a projective surface $X$ of this type, we construct an associated exceptional vector bundle on the nearby fiber $Y$ in the case $H^{2,0}(Y)=H^{1}(Y)=0$.
If $Y=\bP^2$ we show that our construction establishes a bijective correspondence between the possible degenerate surfaces $X$ and exceptional bundles on $Y$ modulo a natural equivalence relation. If $Y$ is of general type then our construction establishes a connection between components of the boundary of the moduli space of surfaces deformation equivalent to $Y$ and exceptional bundles on $Y$.

Let $n,a$ be positive integers such that $a<n$ and $(a,n)=1$.
Consider the cyclic quotient singularity
\begin{eqnarray}\label{wahlsing}
\renewcommand{\arraystretch}{1.5}
\begin{array}{c}
(0 \in \bC^2 /(\bZ/n^2\bZ)),\\
\bZ/n^2\bZ \ni 1 \colon (u,v) \mapsto (\xi u, \xi^{na-1}v), \quad \xi = \exp(2\pi i /n^2).
\end{array}
\end{eqnarray}
We refer to (\ref{wahlsing}) as a \emph{Wahl singularity} of type $\frac{1}{n^2}(1,na-1)$.
A Wahl singularity admits a \emph{$\bQ$-Gorenstein smoothing}, that is, a one parameter deformation such that the general fiber is smooth and the canonical divisor of the total space is $\bQ$-Cartier.
The Milnor fiber of such a smoothing is a rational homology ball. So, if $Y$ is the general fiber of a $\bQ$-Gorenstein smoothing of a surface $X$ with Wahl singularities, then the specialization map $$H_*(Y,\bQ) \rightarrow H_*(X,\bQ)$$ is an isomorphism. For this reason, it is difficult to predict the existence of the degeneration $Y \leadsto X$ given the surface $Y$.

An \emph{exceptional bundle} $F$ on a surface $Y$ is a holomorphic vector bundle such that $\Hom(F,F)=\bC$ and $\Ext^1(F,F)=\Ext^2(F,F)=0$.
In particular $F$ is indecomposable, rigid (no infinitesimal deformations), and unobstructed in families.
So, if $\cY/(0 \in S)$ is a deformation of $Y$ over a germ $(0 \in S)$, then $F$ deforms in a unique way to a family of exceptional bundles on the fibers of $\cY/(0 \in S)$.

\begin{thm} \label{mainthmintro}
Let $X$ be a projective normal surface with a unique singularity $P \in X$ of Wahl type $\frac{1}{n^2}(1,na-1)$.
Let $\cX/(0 \in T)$ be a one parameter deformation of $X$ such that the general fiber $Y$ is smooth and the canonical divisor $K_{\cX}$ of the total space is $\bQ$-Cartier.
\begin{enumerate}
\item Assume that $H_1(Y,\bZ)$ is finite of order coprime to $n$.
Then the specialization map
$$\spz \colon H_2(Y,\bZ) \rightarrow H_2(X,\bZ)$$
is injective with cokernel isomorphic to $\bZ/n\bZ$.
\item Assume in addition that $H^{2,0}(Y)=0$. Then, after a base change $T' \rightarrow T$ of degree $a$, there exists a reflexive sheaf $\cE$ on $\cX':=\cX \times_T T'$ such that
\begin{enumerate}
\item $F:=\cE|_Y$ is an exceptional bundle of rank $n$ on $Y$, and
\item $E:=\cE|_X$ is a torsion-free sheaf on $X$ such that its reflexive hull $E^{\vee\vee}$ is isomorphic to the direct sum of $n$ copies of a reflexive rank $1$ sheaf $A$, and the quotient $E^{\vee\vee}/E$ is a torsion sheaf supported at $P \in X$.
\end{enumerate}
If $\sH$ is a line bundle on $\cX/T$ which is ample on the fibers, then $F$ is slope stable with respect to the the ample line bundle $H:=\sH|_Y$.
Moreover, we have
$$c_1(F)=n c_1(A) \in H_2(Y,\bZ) \subset H_2(X,\bZ),$$
$$c_2(F)=\frac{n-1}{2n}(c_1(F)^2+n+1),$$
$$c_1(F) \cdot K_Y = \pm a \mod n,$$
and
$$H_2(X,\bZ)=H_2(Y,\bZ)+\bZ\cdot(c_1(F)/n).$$
\end{enumerate}
\end{thm}

\begin{rem}
The torsion-free sheaf $E$ on $X$ is a Gieseker semistable limit of the family of stable exceptional bundles $F$ on the fibers of $\cX'/T'$ over $T' \setminus \{0\}$. If $E$ is Gieseker stable, then it is uniquely determined by this property. See \cite[2.B.1]{HL97}.
\end{rem}

\begin{rem}
The exceptional bundles on $Y$ obtained from $F$ by dualizing or tensoring by a line bundle arise from the degeneration $\cX/(0 \in T)$ in the same way.
Indeed, the dual $\cE^{\vee}$ of $\cE$ satisfies the properties \ref{mainthmintro}(2).
Similarly, if $L$ is a line bundle on $Y$, then $L$ extends to a reflexive rank $1$ sheaf $\sL$ on $\cX'$, and the reflexive hull of the tensor product $\cE \otimes \sL$ satisfies the properties \ref{mainthmintro}(2).
\end{rem}

\begin{rem}\label{RR}
The formula for $c_2(F)$ is valid for any exceptional bundle and is given by the Riemann-Roch formula:
$$1=\chi(\cEnd F) = \rk(F)^2\chi(\cO_Y)+(\rk F-1)c_1(F)^2-2\rk(F)c_2(F).$$
\end{rem}

Recent work of Y.~Lee and J.~Park constructs new surfaces of general type with $H^{2,0}=H^1=0$ as $\bQ$-Gorenstein smoothings of rational surfaces with Wahl singularities, see e.g. \cite{LP07}. In these cases our construction produces examples of exceptional bundles on surfaces of general type. As far as I know these are the first such examples. In general little is known about moduli spaces of stable bundles on surfaces of general type unless the expected dimension is large.

\bigskip

\noindent
\emph{Notation}. We work over the complex numbers. In what follows we use the shorthand $\bC^d/\frac{1}{r}(a_1,\ldots,a_d)$ or just $\frac{1}{r}(a_1,\ldots,a_d)$ for the cyclic quotient
\begin{eqnarray*}
\renewcommand{\arraystretch}{1.5}
\begin{array}{c}
\bC^d / (\bZ/r\bZ),\\
\bZ/r\bZ \ni 1 \colon (x_1,\ldots,x_d) \mapsto (\zeta^{a_1}x_1,\ldots,\zeta^{a_d}x_d), \quad \zeta=\exp(2\pi i/r).
\end{array}
\end{eqnarray*}
Some background on reflexive sheaves and toric geometry is reviewed in \S\ref{background}.
\section{Wahl singularities}\label{Wahl}

Let $(P \in X)$ denote the Wahl singularity $(0 \in \bC^2_{u,v}/\frac{1}{n^2}(1,na-1))$.
The canonical divisor $K_X$ is $\bQ$-Cartier of index $n$, that is, $nK_X$ is Cartier and $n \in \bZ_{>0}$ is minimal with this property.
Thus $K_X$ defines a cyclic covering
$$\pi \colon (Q \in Z) \rightarrow (P \in X)$$
of degree $n$, which is unramified over $X \setminus \{P\}$, such that $K_Z=\pi^*K_X$ is Cartier.
The covering $\pi$ is called the \emph{index one cover}.
Explicitly, we have
$$Z = \bC^2_{u,v}/\textstyle{\frac{1}{n}}(1,-1) = (xy=z^n) \subset \bC^3_{x,y,z}$$
where $x=u^n,y=v^n,z=uv$. Thus
$$X = (xy=z^n) \subset \left( \bC^3_{x,y,z}/\textstyle{\frac{1}{n}}(1,-1,a) \right).$$
A smoothing of $P \in X$ is given by
\begin{equation} \label{versalQGdef}
\cX=(xy=z^n+t) \subset \left(\bC^3_{x,y,z}/\textstyle{\frac{1}{n}}(1,-1,a)\right) \times \bC^1_t.
\end{equation}

The link $L$ of the singularity $P \in X$ is the lens space \mbox{$S^3/\frac{1}{n^2}(1,na-1)$}. Let $M$ denote the Milnor fiber of the smoothing (\ref{versalQGdef}), a smooth $4$-manifold with boundary $L$. Wahl observed that $M$ is a rational homology ball.
More precisely, $\pi_1(M)=\bZ/n\bZ$, $H_i(M,\bZ)=0$ for $i>1$, and the map $\pi_1(L) \rightarrow \pi_1(M)$ is the surjection $\bZ/n^2\bZ \rightarrow \bZ/n\bZ$.
To see this, note that by construction $M$ is the quotient of the Milnor fiber $M_Z$ of a smoothing of $Q \in Z$, a Du Val singularity of type $A_{n-1}$, by a free action of $\bZ/n\bZ$. The Milnor fiber $M_Z$ has the homotopy type of a bouquet of $n-1$ copies of $S^2$. So in particular $M_Z$ is simply connected and $\pi_1(M)=\bZ/n\bZ$. Since $M$ is Stein of complex dimension $2$ it has the homotopy type of a cell complex of real dimension $2$. Finally the Euler number $e(M)=e(M_Z)/n=1$, so $b_2(M)=0$.

\begin{rem}
A more explicit analysis yields the following topological description of $M$, see \cite[2.1]{K92}.
Let $N_Z$ be the topological space obtained from $n$ copies $\Delta_j$ of the closed disc $\Delta :=(|z| \le 1) \subset \bC$ by identifying their boundaries.
Define a free action of $\bZ/n\bZ$ on $N_Z$ by
$$\Delta_j \rightarrow \Delta_{j+1}, \quad z \mapsto \zeta z, \quad \zeta=\exp(2\pi i /n),$$
where the indices $j$ are understood modulo $n$.
Let $N$ denote the quotient $N_Z/(\bZ/n\bZ)$. Then $M$ is homotopy equivalent to $N$.
\end{rem}

The $\bQ$-Gorenstein deformations of a quotient singularity are by definition those deformations induced by an equivariant deformation of the index one cover. The deformation (\ref{versalQGdef}) is a versal $\bQ$-Gorenstein deformation, that is, every $\bQ$-Gorenstein deformation of $(P \in X)$ is obtained from (\ref{versalQGdef}) by pullback. A one parameter smoothing of a quotient singularity is $\bQ$-Gorenstein iff the canonical divisor of the total space is $\bQ$-Cartier \cite[3.4]{H04}.

The $\bQ$-Gorenstein condition is natural from the point of view of Mori theory and is used in the definition of the compactification $\ocM$ of the moduli space of surfaces of general type analogous to the Deligne--Mumford compactification of the moduli space of curves \cite{KSB88}. In particular, since the versal $\bQ$-Gorenstein deformation space of a Wahl singularity is a smooth curve germ, if $X$ is a normal projective surface such that $K_X$ is ample, $X$ has a unique singularity of Wahl type, and $H^2(T_X)=0$ (so that there are no local-to-global obstructions for deformations of $X$), then $\ocM$ is smooth near $[X]$ and locally trivial deformations of $X$ determine a codimension one component of the boundary of $\ocM$.

\section{Blowup construction}

\begin{prop}\label{blowup}
Let $n$ and $a$ be positive integers such that $a<n$ and $(a,n)=1$.
Let $(P \in \cX)/T$ be a one parameter $\bQ$-Gorenstein smoothing of a Wahl singularity $(P \in X) \simeq (0 \in  \bC^2_{u,v} / \frac{1}{n^2}(1,na-1) )$.
Then, after a base change $T' \rightarrow T$ of degree $a$, there exists a birational morphism $\pi \colon \tilde{\cX} \rightarrow \cX'$ satisfying the following properties.
\begin{enumerate}

\item The locus $W:=\pi^{-1}(P) \subset \tilde{\cX}$ is a normal surface isomorphic to the weighted projective hypersurface
$$(XY=Z^n+T^a) \subset \bP(1,na-1,a,n).$$
\item The morphism $\pi$ restricts to an isomorphism $\tilde{\cX}\setminus W \rightarrow \cX \setminus \{P\}$.
\item The special fiber $\tilde{X}:=\tilde{\cX}_0$ is reduced and is the union of two components $\tilde{X}_1$ and $\tilde{X}_2$ meeting along a smooth rational curve $C$, where $\tilde{X}_1$ is the strict transform of $X$ and $\tilde{X_2}=W$ is the exceptional divisor.
The curve $C \subset \tilde{X}_1$ is the exceptional curve of the restriction $p \colon \tilde{X}_1 \rightarrow X$ of $\pi$, and $C=(T=0) \subset W$.
\item Let $Q \in C \subset \tilde{X}$ denote the point with homogeneous coordinates $(0 \colon 1 \colon 0 \colon 0)$.
The reducible surface $\tilde{X}$ has normal crossing singularities $(xy=0) \subset \bC^3_{x,y,z}$ along $C \setminus \{Q\}$, an orbifold normal crossing singularity $(xy=0) \subset \bC^3_{x,y,z}/ \frac{1}{na-1}(1,-1,a^2)$ at $Q$, and is smooth elsewhere.
\end{enumerate}
Moreover, the birational morphism $p \colon \tilde{X}_1 \rightarrow X$ is the weighted blowup of $P \in X$ with weights $\frac{1}{n^2}(1,na-1)$ with respect to the orbifold coordinates $u,v$.
\end{prop}
\begin{proof}
Recall from \S\ref{Wahl} that the versal $\bQ$-Gorenstein deformation of $(P \in X)$ is given by
$$\cX =(xy=z^n+t) \subset \left( \bC^3_{x,y,z} / \textstyle{\frac{1}{n}}(1,-1,a) \right) \times \bC^1_t.$$
We describe the construction of $\pi$ for the versal deformation. In general we obtain the morphism by pullback from the versal case.

We make the base change $t \mapsto t^a$ and blowup $(x,y,z,t)$ with weights $w=\frac{1}{n}(1,na-1,a,n)$ to obtain the desired birational morphism.
(See \S\ref{wtdblowup} for background on weighted blowups.)
Let $f \colon \tilde{A} \rightarrow A$ denote the blowup of the ambient space $A:=\frac{1}{n}(1,-1,a) \times \bC^1_t$.
Then $f$ has exceptional divisor $E=\bP(1,na-1,a,n)$ with weighted homogeneous coordinates $X,Y,Z,T$
corresponding to the orbifold coordinates $x,y,z,t$ at $0 \in A$. The $3$-fold $\tilde{\cX}$ is by definition the strict transform of $\cX \subset A$ under the map $f$. Observe that the equation $(xy=z^n+t^a) \subset A$ of $\cX$ is homogeneous with respect to the weight vector $w$. It follows that the exceptional locus
$$W:=\pi^{-1}(P) = E \cap \tilde{\cX}$$
of $\pi$ is given by the same equation in the weighted projective space $E$.
Now consider the fiber $\tilde{X}=\tilde{\cX}_0$ of $\tilde{\cX}$ over $0 \in T=\bC^1_t$.
Observe that the weight of $t$ is equal to $1$ (we made the base change above to ensure this).
It follows that the Cartier divisor $\tilde{X}=(t=0) \subset \tilde{\cX}$ is reduced, equal to the sum $X'+W$ of the strict transform of $X$ and the exceptional divisor $W$. It is easy to check in the charts for $\tilde{A}$ that the singularities of $\tilde{X}$ are as described in the statement.

Finally, since $X=(t=0) \subset A$ is identified with $\bC^2_{u,v}/\frac{1}{n^2}(1,na-1)$ by $(u,v) \mapsto (x,y,z)=(u^n,v^n,uv)$, we see that the induced birational morphism $p \colon \tilde{X}_1 \rightarrow X$ is the blowup with weights $\frac{1}{n^2}(1,na-1)$ with respect to $u,v$.
We remark that, since $P \in X$ is a cyclic quotient singularity, its minimal resolution $\hat{X} \rightarrow X$ has exceptional locus a nodal chain
$F=F_1+\ldots+F_r$ of smooth rational curves, such that the strict transforms $(u=0)'$ and $(v=0)'$ of the coordinate axes intersect the end components $F_1$ and $F_r$ respectively. Then $\tilde{X}_1$ is obtained from $\hat{X}$ by contracting the chain $F_1+\cdots+F_{r-1}$ of exceptional curves disjoint from $(v=0)'$ to a cyclic quotient singularity of type $\frac{1}{na-1}(a^2,-1)$.
\end{proof}

\section{Glueing}

Let $\cX/(0 \in T)$ be a one parameter deformation of a projective normal surface $X$ with quotient singularities.
Let $P \in X$ be a Wahl singularity of type $\frac{1}{n^2}(1,na-1)$ such that the germ $(P \in \cX)/T$ is a $\bQ$-Gorenstein smoothing of $(P \in X)$.
Let $P_1=P,P_2,\ldots,P_r$ be the singularities of $X$ and $L_i$ the link of the singularity $P_i \in X$.
Let $Y$ denote a general fiber of $\cX/T$. In this section we make the following assumptions:

\begin{enumerate}
\item The map
\begin{equation}\label{classgplocalglobal}
H_2(X,\bZ) \rightarrow \bigoplus H_1(L_i,\bZ), \quad \alpha \mapsto (\alpha \cap L_i)
\end{equation}
is surjective.
\item We have $H^{2,0}(Y)=0$ and $H^1(Y,\bZ)=0$.
\end{enumerate}

Let $(0 \in T') \rightarrow (0 \in T)$ and $\pi \colon \tilde{\cX} \rightarrow \cX'$ be the base change and blowup of Proposition~\ref{blowup}.

\begin{lem}\label{linebundle}
There exists a line bundle $\tilde{A}$ on the strict transform $\tilde{X}_1 \subset \tilde{\cX}$ of $X$ such that the restriction of $\tilde{A}$ to the exceptional curve $C$ of $p \colon \tilde{X}_1 \rightarrow X$ has degree $1$.
\end{lem}
\begin{proof}
By our assumptions $H^i(\cO_Y)=0$ for $i>0$. Since $X$ has quotient singularities, it follows that $H^i(\cO_{\tilde{X}_1})=H^i(\cO_X)=0$ for $i>0$ \cite[4.6,5.3]{DB81}.
Hence $c_1 \colon \Pic \tilde{X}_1 \rightarrow H^2(\tilde{X}_1,\bZ)$ is an isomorphism, and we must show that the restriction map $H^2(\tilde{X}_1,\bZ) \rightarrow H^2(C,\bZ)$ is surjective.

Let $B_i$ denote the intersection of $X$ with a small ball about $P_i$ in some embedding.
Then $B_i$ is topologically the cone over the link $L_i$ of $P_i \in X$, in particular $B_i$ is contractible.
Write $X^o = X \setminus \bigcup B_i$. Define $N:= p^{-1} B_1 \subset \tilde{X}_1$. Then $C \subset N$ is a deformation retract.
The Mayer-Vietoris sequence for $\tilde{X}_1=X^o \cup (N \cup B_2 \cup \cdots \cup B_r)$ gives an exact sequence
$$H^2(\tilde{X}_1,\bZ) \rightarrow H^2(X^o,\bZ) \oplus H^2(C,\bZ) \rightarrow \bigoplus H^2(L_i,\bZ)$$
The restriction map $H^2(X^o,\bZ) \rightarrow \bigoplus H^2(L_i,\bZ)$ is identified with the map (\ref{classgplocalglobal}) by Poincar\'e duality (cf. \S\ref{reflexive}),
and (\ref{classgplocalglobal}) is surjective by assumption. It follows that the restriction map $H^2(\tilde{X}_1,\bZ) \rightarrow H^2(C,\bZ)$ is surjective as required.
\end{proof}

\begin{prop}\label{glueing}
Suppose $G$ is an exceptional bundle of rank $n$ on the $\pi$-exceptional divisor $W$ such that $G|_C \simeq \cO_C(1)^{\oplus n}$.
Let $\tilde{E}$ be the vector bundle on the reducible surface $\tilde{X}$ obtained by glueing $\tilde{A}^{\oplus n}$ on $\tilde{X}_1$ and $G$ on $\tilde{X}_2=W$ along $\cO_C(1)^{\oplus n}$ on $C$. Then $\tilde{E}$ is an exceptional vector bundle on $\tilde{X}$.

Let $\tilde{\cE}$ denote the vector bundle on $\tilde{\cX}$ obtained by deforming $\tilde{E}$.
Let $\cE:=(\pi_*\tilde{\cE})^{\vee\vee}$ be the reflexive hull of the pushforward of $\tilde{\cE}$ to $\cX'$.
Then $\cE|_{\cX'_t}$ is an exceptional vector bundle on $\cX'_t$ for $t \neq 0$ and $E:=\cE|_X$ is a torsion-free sheaf on $X$.

Let $A:=(p_*\tilde{A})^{\vee\vee}$ be the reflexive hull of the pushforward of $\tilde{A}$ to $X$.
Then the reflexive hull $E^{\vee\vee}$ of $E$ equals $A^{\oplus n}$ and $E^{\vee\vee}/E$ is a torsion sheaf supported at $P \in X$.
\end{prop}

\begin{rem}
We show in Proposition~\ref{Wbdles} that such bundles $G$ exist by induction on $a$. The induction step uses Proposition~\ref{glueing}.
\end{rem}

\begin{proof}
Since $\tilde{X}=\tilde{X}_1 \cup \tilde{X}_2$ has orbifold normal crossing singularities, we have an exact sequence of sheaves on $\tilde{X}$
$$0 \rightarrow \cO_{\tilde{X}} \rightarrow \cO_{\tilde{X}_1} \oplus \cO_{\tilde{X}_2} \rightarrow \cO_C \rightarrow 0.$$
The sheaf $\tilde{E}$ on $\tilde{X}$ is defined as the cokernel of the map
$$\tilde{A}^{\oplus n} \oplus G \rightarrow \cO_C(1)^{\oplus n}$$
where we fix isomorphisms $\tilde{A}^{\oplus n}|_C \simeq \cO_C(1)^{\oplus n} \simeq G|_C$.
(Note that the isomorphism type of $\tilde{E}$ does not depend on the choice of these isomorphisms because
$\Aut(\tilde{A}^{\oplus n}) = \Aut(\cO_C(1)^{\oplus n})=\GL(n,\bC)$.)
Then $\tilde{E}$ is a vector bundle on $\tilde{X}$ such that $\tilde{E}|_{\tilde{X}_1} \simeq \tilde{A}^{\oplus n}$ and $\tilde{E}|_{\tilde{X}_2} \simeq G$.

We show that $\tilde{E}$ is exceptional. Consider the exact sequence
$$0 \rightarrow \cEnd \tilde{E} \rightarrow \cEnd \tilde{E} |_{\tilde{X}_1} \oplus \cEnd \tilde{E} |_{\tilde{X}_2} \rightarrow \cEnd \tilde{E}|_C \rightarrow 0,$$
which is equal to
$$0 \rightarrow \cEnd \tilde{E} \rightarrow \cO_{\tilde{X}_1}^{n \times n} \oplus \cEnd G \rightarrow \cO_C^{n \times n} \rightarrow 0.$$
The curve $C$ is smooth and rational, so $H^1(\cO_C)=0$, and $H^i(\cO_{\tilde{X}_1})=0$ for $i>0$ as noted in the proof of Lemma~\ref{linebundle} above. We deduce that $H^i(\cEnd \tilde{E} ) = H^i(\cEnd G)$. Hence $\tilde{E}$ is exceptional because $G$ is exceptional.

The vector bundle $\tilde{E}$ deforms uniquely to a vector bundle $\tilde{\cE}$ over $\tilde{\cX}$ because $\tilde{E}$ is exceptional.
The restrictions of $\tilde{\cE}$ to the fibers of $\tilde{\cX}/T$ are exceptional by upper semicontinuity of cohomology.

Note that $E:=\cE|_X$ is torsion-free because $\cE$ is reflexive.
Indeed, $\cE$ satisfies Serre's condition $S_2$ and $X=(t=0) \subset \cX$ is a Cartier divisor. So the restriction $E=\cE|_{X}$ satisfies $S_1$, that is, $E$ is torsion-free.

By construction $E|_{X \setminus \{P\}}= A^{\oplus n}|_{X \setminus \{P\}}$, so $E^{\vee\vee}=A^{\oplus n}$ and $E^{\vee\vee}/E$ is supported at $P$.
\end{proof}

\begin{prop}\label{stability}
Suppose $\sH$ is a line bundle on $\cX/T$ which is ample on fibers.
Then for $t \neq 0$ the exceptional vector bundle $\tilde{\cE}|_{\cX'_t}$ on $\cX'_t$ constructed in Proposition~\ref{glueing} is slope stable with respect to $\sH|_{\cX'_t}$.
\end{prop}
\begin{proof}
Let $\sH'$ be the pullback of $\sH$ to $\cX'$, and define
$$\tilde{\sH}:=\pi^*{\sH'}^{\otimes N} \otimes \cO_{\tilde{\cX}}(-(na-1)W)$$
for $N \gg 0$, where $W=\tilde{X}_2$ is the exceptional divisor of $\pi$.
Then $\tilde{\sH}$ is a line bundle on $\tilde{\cX}/T'$ which is ample on fibers, and its restriction to $\cX'_t$ for $t \neq 0$ coincides with the restriction of $\sH^{\otimes N}$.
In what follows we write $\mu(F)$ for the slope of a sheaf $F$ on a surface $S \subset \tilde{\cX}$ defined using the polarization on $S$ given by the restriction of $\tilde{\sH}$, that is,
$$\mu(F)= \deg F / \rk F := (c_1(F) \cdot \tilde{\sH}|_S) /\rk(F)$$

Suppose $\tilde{\cE}|_{\cX'_t}$ is not slope stable with respect to $\sH|_{\cX'_t}$ for $t \neq 0$.
Then, by the argument for openness of stability \cite[2.3.1]{HL97}, after a finite surjective basechange (which we suppress in our notation),
there is a surjection $\tilde{\cE} \rightarrow \cC$ such that $\cC$ is flat over $T'$, $0 < \rk (\cC) < \rk(\tilde{\cE})$, and
$\mu(\cC|_{\cX'_t}) \le \mu(\tilde{\cE}|_{\cX'_t})$ for all $t \in T'$.
(Note that, by flatness of $\cC$ over $T$, $\cC|_{\tilde{X}}$ is a sheaf of constant rank on the reducible surface $\tilde{X}$.
Thus $\mu(\cC|_{\tilde{X}})$ is well-defined.)

Let $\tilde{E} \rightarrow C$ and $\tilde{E}_i \rightarrow C_i$ denote the restrictions of $\tilde{\cE} \rightarrow \cC$ to $\tilde{X}$ and $\tilde{X}_i$ for $i=1,2$.
Recall that $\tilde{E}_1$ is the direct sum of $n$ copies of the line bundle $\tilde{A}$ and $\tilde{E}_2$ is the exceptional vector bundle $G$ on $W$.
In particular $\tilde{E}_1$ is slope semistable, and $\tilde{E}_2$ is slope stable by Proposition~\ref{Gstable} below.
Thus $\mu(C_1) \ge \mu(\tilde{E}_1)$ and $\mu(C_2) > \mu(\tilde{E}_2)$.
We deduce that $\mu(C)=\sum \mu(C_i) > \mu(\tilde{E})=\sum \mu(\tilde{E}_i)$, a contradiction.
\end{proof}

\section{Localized exceptional bundles}

\begin{prop}\label{Wbdles}
Let $n,a$ be positive integers such that $a<n$ and $(a,n)=1$. Write
$$W=W_{n,a}:=(XY=Z^n+T^a) \subset \bP(1,na-1,a,n)$$
Let $C_1$ and $C_2$ be the smooth rational curves on $W$ defined by $C_1=(Z=0)$ and $C_2=(T=0)$.
Then there exist exceptional vector bundles $F_1$ and $F_2$ on $W$ of ranks $a$ and $n$ such that for each $j=1,2$ we have
\begin{enumerate}
\item $H^i(F_j^{\vee})=0$ for all $i$,
\item $H^i(F_j)=0$ for $i>0$,
\item $F_j$ is generated by global sections, and
\item $F_j|_{C_j} \simeq \cO_{C_j}(1)^{\oplus \rk F_j}$.
\end{enumerate}
(Here $\cO_{C_j}(1)$ denotes the line bundle of degree $1$ on the smooth rational curve $C_j$.)
\end{prop}

\begin{con}\label{Wdegen}
The proof of Proposition~\ref{Wbdles} uses the following degeneration of $W$.
Consider the one parameter family of normal surfaces
$$\cX=(XY=tZ^n+T^a) \subset \bP(1,na-1,a,n) \times \bC^1_t.$$
Note that $\cX_t \simeq W$ for $t \neq 0$. The special fiber $X:=\cX_0$ is isomorphic to the weighted projective plane $\bP(1,na-1,a^2)$ via the morphism
$$\bP(1,na-1,a^2) \rightarrow X=(XY=T^a) \subset \bP(1,na-1,a,n),$$
$$(U,V,W) \mapsto (X,Y,Z,T)=(U^a , V^a , W , UV).$$
The surface $X$ has two singular points $P=(0 \colon 0 \colon 1)$ and $Q=(0 \colon 1 \colon 0)$.
The point $P$ is a Wahl singularity of type $\frac{1}{a^2}(1,ab-1)$ where $b = n \bmod a$, and the germ $(P \in \cX)/(0 \in \bC^1_t)$ is a $\bQ$-Gorenstein smoothing of $(P \in X)$. The point $Q$ is a cyclic quotient singularity of type $\frac{1}{na-1}(1,a^2)$ and the deformation $\cX/(0 \in \bC^1_t)$ is locally trivial near $Q$.
\end{con}

\begin{proof}[Proof of Proposition~\ref{Wbdles}]
We first construct $F_2$ given $F_1$.
The vector bundle $F_1$ is globally generated, that is, the natural morphism
\begin{equation} \label{gs}
H^0(F_1) \otimes \cO_W \rightarrow F_1
\end{equation}
is surjective.
We define $F_2$ as the dual of the kernel of (\ref{gs}).
Thus $F_2$ is a vector bundle such that  $\rk(F_2)=h^0(F_1)-\rk(F_1)$ and $c_1(F_2)=c_1(F_1)$.
The bundle $F_2$ is exceptional and $H^i(F_2^{\vee})=0$ for all $i$ by \cite[\S2.4]{G90}.
Indeed, in the terminology of op. cit., the pair $(\cO_W,F_1)$ is an \emph{exceptional pair}, and $(F_2^{\vee},\cO_W)$ is its \emph{left mutation}.
Moreover, the exact sequence
$$0 \rightarrow F_1^{\vee} \rightarrow H^0(F_1)^* \otimes \cO_W \rightarrow F_2 \rightarrow 0$$
shows that $F_2$ is globally generated and $H^i(F_2)=0$ for $i > 0$.

By Lemma~\ref{classgroup} the homology group $H_2(W,\bZ)$ is isomorphic to $\bZ$, generated by the restriction $H$ of the positive generator $A:=c_1(\cO_{\bP}(1))$ of $H_4(\bP,\bZ)$, where $\bP=\bP(1,na-1,a,n)$ denotes the ambient weighted projective space.
(See \S\ref{wps} for background on weighted projective spaces.)
Note that $H^2=1/(na-1)$ (because $H^2= A^2 \cdot W = (na) A^3$ and $A^3=1/((na-1)an)$).
Now $C_1 =(Z=0) \sim aH$ and $F_1|_{C_1} \simeq \cO_{C_1}(1)^{\oplus a}$, thus $c_1(F_1) \cdot aH=a$ and so $c_1(F_1)=(na-1)H$.
Since $F_1$ is an exceptional bundle of rank $a$ we have
$$c_2(F_1)=\frac{a-1}{2a}(c_1(F_1)^2+a+1),$$
see Remark~\ref{RR}.
The canonical class $K_W$ of $W$ is given by the adjunction formula
$$K_W=(K_{\bP}+W)|_W=-(a+n)H.$$
Now the Riemann--Roch formula gives
$$h^0(F_1)=\chi(F_1)=a\chi(\cO_W)+\frac{1}{2}c_1(F_1)(c_1(F_1)-K_W)-c_2(F_1)=a+n.$$
Thus $F_2$ is a vector bundle of rank $n$.

It remains to show that $F_2|_{C_2} \simeq \cO_{C_2}(1)^{\oplus n}$.
The bundle $F_2|_{C_2}$ on $C_2 \simeq \bP^1$ has rank $n$ and degree
$$c_1(F_2) \cdot C_2 = c_1(F_1) \cdot n H = (na-1)nH^2 = n.$$
So it suffices to show that $F_2|_{C_2}$ is rigid, that is, $H^1(\cEnd F_2|_{C_2})=0$.
Consider the exact sequence
$$0 \rightarrow \cEnd F_2 (-C_2) \rightarrow \cEnd F_2 \rightarrow \cEnd F_2|_{C_2} \rightarrow 0.$$
We have $H^1(\cEnd F_2)=0$ and $$H^2(\cEnd F_2(-C_2))=H^0(\cEnd F_2(K_W+C_2))^*=H^0(\cEnd F_2(-aH))^*=0$$
because $K_W+C_2 \sim -aH < 0$ and $\End F_2 = \bC$. So $H^1(\cEnd F_2|_{C_2})=0$ as required.

We now prove the existence of $F_1$ by induction on $a$.
If $a=1$ then $W=\bP(1,n-1,1)$ and we can take $F_1=\cO_W(n-1)$. Now suppose $a>1$.
Consider the degeneration $\cX/T$ of $W$ described in Construction~\ref{Wdegen}.
Write $n=ka+b$, $0<b<a$.
The point $P=(0 \colon 0 \colon 1) \in X=\bP(1,na-1,a^2)$ is a Wahl singularity of type $\frac{1}{a^2}(1,ab-1)$, and the germ $(P \in \cX)/T$ is a $\bQ$-Gorenstein smoothing of $(P \in X)$.
Let $(0 \in T') \rightarrow (0 \in T)$ and $\pi \colon \tilde{\cX} \rightarrow \cX'$ be the basechange and blowup of Proposition~\ref{blowup}.
The exceptional divisor of $\pi$ is $\tilde{X}_2=W_{a,b}$. By induction and the construction of $F_2$ from $F_1$ above, there exists a vector bundle $G$ on $\tilde{X}_2$ of rank $a$ such that $H^i(G^{\vee})=0$ for all $i$, $H^i(G)=0$ for $i>0$, $G$ is globally generated, and $G|_C \simeq \cO_C(1)^{\oplus a}$, where $C = \tilde{X}_1 \cap \tilde{X_2}$ is the double curve of $\tilde{X}=\tilde{\cX}_0$.

The surface $X=\bP(1,na-1,a^2)$ is the toric variety associated to a free abelian group $N$ and a fan $\Sigma$ in $N \otimes_{\bZ} \bR$ as follows.
The group $N \simeq \bZ^2$ is generated by vectors $v_0,v_1,v_2$ satisfying the relation $$v_0+(na-1)v_1+a^2v_2=0.$$
The fan $\Sigma$ is the complete fan with rays generated by $v_0,v_1,v_2$.
The birational morphism $p \colon \tilde{X}_1 \rightarrow X$ is the weighted blowup of the point $P =(0 \colon 0 \colon 1) \in X$ with weights $\frac{1}{a^2}(1,ab-1)$
with respect to the orbifold coordinates $u=U/W,v=V/W$.
The morphism $p$ corresponds to the refinement $\tilde{\Sigma}$ of the fan $\Sigma$ obtained by adding the ray generated by
$$w:=\textstyle{\frac{1}{a^2}}(v_0+(ab-1)v_1) \in N.$$

Let $D$ be the divisor $(V=0) \subset X =\bP(1,na-1,a^2)$ and $D' \subset \tilde{X}_1$ its strict transform.
Note that $D' \subset \tilde{X}_1$ is the toric boundary divisor corresponding to the ray $\bR_{\ge 0} \cdot v_1$ of $\tilde{\Sigma}$.
The divisor $D'$ is Cartier and $D' \cdot C = 1$.
By Proposition~\ref{glueing} there is an exceptional vector bundle $\tilde{E}$ on $\tilde{X}$ obtained by glueing $\cO_{\tilde{X}_1}(D')^{\oplus n}$ on $\tilde{X}_1$ and $G$ on $\tilde{X}_2$ along $\cO_C(1)^{\oplus a}$ on $C$, and $\tilde{E}$ deforms to an exceptional bundle $F_1$ on the general fiber $W$ of $\tilde{\cX}/T$. It remains to show that $F_1$ satisfies the properties (1)--(4) in the statement.

(1,2) It suffices to verify the corresponding vanishings for $\tilde{E}$.
We have the exact sequence of sheaves on $\tilde{X}$
$$0 \rightarrow \tilde{E}^{\vee} \rightarrow \cO_{\tilde{X}_1}(-D')^{\oplus a} \oplus G^{\vee} \rightarrow \cO_C(-1)^{\oplus a} \rightarrow 0.$$
Now $H^i(G^{\vee})=0$ for all $i$ by assumption, $H^i(\cO_C(-1))=0$ for all $i$,
and we find $H^i(\cO_{\tilde{X_1}}(-D'))=0$ for all $i$ using the exact sequence
$$0 \rightarrow \cO_{\tilde{X_1}}(-D') \rightarrow \cO_{\tilde{X}_1} \rightarrow \cO_{D'} \rightarrow 0.$$
Hence $H^i(\tilde{E}^{\vee})=0$ for all $i$. Similarly, for (2) we consider the exact sequence
\begin{equation} \label{ncdec}
0 \rightarrow \tilde{E} \rightarrow \cO_{\tilde{X}_1}(D')^{\oplus a} \oplus G \rightarrow \cO_C(1)^{\oplus a} \rightarrow 0.
\end{equation}
A toric calculation shows that $(D')^2=k>0$, where $n=ka+b$ as above. (Indeed, with notation as above, we compute $v_2+w=-kv_1$, so $(D')^2=k$ by \cite{F93}, p. 44.)
The exact sequence
$$0 \rightarrow \cO_{\tilde{X}_1} \rightarrow \cO_{\tilde{X}_1}(D') \rightarrow \cO_{D'}(k) \rightarrow 0$$
shows $H^i(\cO_{\tilde{X}_1}(D'))=0$ for $i>0$.
The exact sequences
$$0 \rightarrow \cO_{\tilde{X}_1}(-C) \rightarrow \cO_{\tilde{X}_1}(D'-C) \rightarrow \cO_{D'}(k-1) \rightarrow 0$$
and
$$0 \rightarrow \cO_{\tilde{X}_1}(-C) \rightarrow \cO_{\tilde{X}_1} \rightarrow \cO_C \rightarrow 0$$
show that $H^1(\cO_{\tilde{X}_1}(D'-C))=0$, hence the restriction map $H^0(\cO_{\tilde{X}_1}(D')) \rightarrow H^0(\cO_C(1))$ is surjective.
We deduce that $H^i(\tilde{E})=0$ for $i>0$.

(3) Since $H^1(\tilde{E})=0$ it suffices to prove that $\tilde{E}$ is globally generated.
The bundle $G$ is globally generated by assumption and the restriction map $H^0(\cO_{\tilde{X}_1}(D')) \rightarrow H^0(\cO_{C}(1))$ is surjective as proved above.
It follows that $\tilde{E}$ is globally generated along $\tilde{X}_2$.

Let $S$ denote the cokernel of the natural map
$$H^0(\tilde{E}) \otimes \cO_{\tilde{X}_1} \rightarrow \tilde{E}|_{\tilde{X}_1}.$$
Thus $\tilde{E}$ is globally generated iff $S=\emptyset$.
We first show that $S$ is a union of toric strata of the toric surface $\tilde{X}_1$.
The action $\sigma$ of the big torus $T \simeq (\bC^*)^2$ on $\tilde{X}_1$ restricts to an action on $C$
given by
$$C=(T=0) \subset \tilde{X}_2=(XY=Z^a+T^b) \subset \bP(1,ab-1,b,a)$$
$$T \ni t \colon (X , Y , Z) \mapsto (X , \chi(t)^aY , \chi(t)Z)$$
for some character $\chi$ of $T$.
Consider the action $\sigma'$ of $T$ on $\tilde{X}_1$ given by $\sigma'(t,x)=\sigma(t^b,x)$.
Then $\sigma'$ extends to an action of $T$ on $\tilde{X}$ defined on $\tilde{X}_2$ by
$$t \colon (X , Y , Z , T) \mapsto (X , \chi(t)^{ab}Y , \chi(t)^bZ , \chi(t)^aT).$$
Since $\tilde{E}$ is exceptional, it has no nontrivial deformations, so $t^*\tilde{E}$ is isomorphic to $\tilde{E}$ for all $t \in T$.
It follows that $S \subset \tilde{X}_1$ is a union of toric strata. Now $\tilde{E}$ is globally generated along $C$, so it suffices to show that $\tilde{E}$ is globally generated along the toric boundary divisor $B \subset \tilde{X}_1$ disjoint form $C$.
Note that $B$ maps isomorphically to $(W=0) \subset X$ under $p$.

The divisor $D'$ intersects each of $B$ and $C$ transversely in a smooth point of $\tilde{X}_1$.
In particular, $\tilde{E}|_B \simeq \cO_{B}(1)^{\oplus a}$.
Let $f \colon \tilde{X}_1' \rightarrow \tilde{X}_1$ be the blowup of  $k$ distinct interior points of $D'$ so that the strict transform $D''$ satisfies $D'' = f^*D' - \sum_{i=1}^k E_i$ where the $E_i$ are the exceptional curves. Then $(D'')^2=0$ and so $D''$ defines a ruling of $\tilde{X}_1'$ with sections $B$ and $C$. So $H^0(\cO_{\tilde{X}_1'}(D'')) \subset H^0(\cO_{\tilde{X}_1}(D'))$ maps isomorphically to $H^0(\cO_{B}(1))$ and $H^0(\cO_{C}(1))$. Since $G=\tilde{E}|_{\tilde{X}_2}$ is globally generated along $C$ we deduce that $\tilde{E}$ is globally generated along $B$.

(4) The divisor $\cB:=(Z=0) \subset \cX/\bC^1_t$ is a $\bP^1$-bundle over the base with fiber $C_2 \subset W = \cX_t$ for $t \neq 0$ and $B=(W=0) \subset X$ for $t=0$. As noted above we have $\tilde{E}|_B \simeq \cO_B(1)^{\oplus a}$, so also $F_1|_{C_2} \simeq \cO_{C_2}(1)^{\oplus a}$.
\end{proof}

\begin{lem}\label{classgroup}
Let $W=(XY=Z^n+T^a) \subset \bP(1,na-1,a,n)$. Then $H_2(W,\bZ)$ is isomorphic to $\bZ$, generated by the restriction of the positive generator $A:=c_1(\cO_{\bP}(1)) \in H_4(\bP,\bZ)$ of the homology of the ambient weighted projective space $\bP=\bP(1,na-1,a,n)$.
\end{lem}
\begin{proof}
By Construction~\ref{Wdegen} the surface $W$ is obtained from the weighted projective plane $X=\bP(1,na-1,a^2) \subset \bP$ by smoothing the Wahl singularity $P \in X$ of type $\frac{1}{a^2}(1,ab-1)$. We have $A|_X=aL$, where $L$ denotes the positive generator of $H_2(X,\bZ)$.
Now by Lemma~\ref{MV} below we deduce that $H_2(W,\bZ)$ is generated by $A|_W$ as required.
\end{proof}

\begin{lem} \label{MV}
Let $X$ be a compact normal surface and $P \in X$ a Wahl singularity of type $\frac{1}{n^2}(1,na-1)$.
Let $L$ denote the link of $P \in X$. Assume that $H_2(X,\bZ) \rightarrow H_1(L,\bZ)$ is surjective.
Let $\cX/(0 \in T)$ be a deformation of $X$ over a smooth curve germ $(0 \in T)$ such that the germ $(P \in \cX)/(0 \in T)$ is a $\bQ$-Gorenstein smoothing of $P \in X$ and the deformation $\cX/T$ is locally trivial elsewhere. Let $Y$ denote a general fiber of $\cX/T$.
Then the specialization map
$$\spz \colon H_2(Y,\bZ) \rightarrow H_2(X,\bZ)$$
is injective with cokernel isomorphic to $\bZ/n\bZ$.
\end{lem}
\begin{proof}
Let $B \subset X$ be the intersection of $X$ with a small ball about $P$ in some embedding, and $X^o$ the complement of $B$.
Then $B$ is contractible and has boundary $L$, the link of $P \in X$.
Let $M \subset Y$ be the Milnor fiber of the smoothing of $P \in X$.
Then the Mayer--Vietoris sequences for $Y=X^o \cup M$ and $X=X^o \cup B$ give a commutative diagram with exact rows
$$
\begin{CD}
0 @>>> H_2(X^o) @>>> H_2(Y) @>>> H_1(L) @>>> H_1(X^o) \oplus H_1(M)\\
@. 	@| 	     @VVV 	   @|           @VVV\\
0 @>>> H_2(X^o) @>>> H_2(X) @>>> H_1(L) @>>> H_1(X^o) \\
\end{CD}
$$
using $H_2(L)=H_2(M)=0$ and contractibility of $B$. Now $H_2(X) \rightarrow H_1(L)$ is surjective by assumption and $H_1(L) \rightarrow H_1(M)$ is a surjection of the form $\bZ/n^2\bZ \rightarrow \bZ/n\bZ$.
It follows that $H_2(Y) \rightarrow H_2(X)$ is injective with cokernel $H_1(M)$ isomorphic to $\bZ/n\bZ$ as claimed.
\end{proof}

\begin{prop}\label{Gstable}
Let $W$ be the surface of Proposition~\ref{Wbdles} and $G$ an exceptional vector bundle on $W$.
Then $G$ is slope stable.
\end{prop}
\begin{proof}
Note first that the choice of polarization is irrelevant because $b_2(W)=1$, see Lemma~\ref{classgroup}.
Let $D \in |-K_W|$ be a general member of the anticanonical linear system on $W$.
Then $D$ is a rational nodal curve of arithmetic genus one.
Consider the exact sequence
$$0 \rightarrow \cEnd G(-D) \rightarrow \cEnd G \rightarrow \cEnd G|_D \rightarrow 0.$$
We have $H^i(\cEnd G)=0$ for $i \neq 0$ because $G$ is exceptional, and
$$H^i(\cEnd G(-D))=H^i(\cEnd G(K_W))=H^{2-i}(\cEnd G)^* = 0$$
for $i \neq 2$ by Serre duality.
Thus $H^0(\cEnd G|_D)=H^0(\cEnd G) = \bC$, that is, $G|_D$ is simple.
Now by \cite[4.13]{BK06} $G|_D$ is slope stable. It follows that $G$ is slope stable.
\end{proof}

\begin{proof}[Proof of Theorem~\ref{mainthmintro}]
By Propositions~\ref{glueing} and \ref{Wbdles}, to establish the existence of the sheaf $\cE$ satisfying \ref{mainthmintro}(2)(a,b) it suffices to show that the hypotheses of Theorem~\ref{mainthmintro} imply the hypothesis (1) of \S4, namely, that the map
\begin{equation}\label{restriction}
H_2(X,\bZ) \rightarrow H_1(L,\bZ)
\end{equation}
is surjective, where $L$ is the link of the singularity.

Let $B$ denote the intersection of $X$ with a small ball centered at the singularity $P \in X$ in some embedding, $M$ the Milnor fiber of the smoothing, and write $X^o:=X \setminus B$.
As in the proof of Lemma~\ref{MV}, consider the Mayer-Vietoris sequences for $Y=X^o \cup M$ and $X=X^o \cup B$.
Let $I_X$ and $I_Y$ denote the image of $H_2(X)$ and $H_2(Y)$ in $H_1(L)$.
We obtain a commutative diagram with exact rows
$$
\begin{CD}
0 @>>> H_1(L)/I_Y @>>> H_1(X^o) \oplus H_1(M) @>>> H_1(Y) @>>> 0\\
@. 	     @VVV 	       @VVV         @VVV\\
0 @>>> H_1(L)/I_X @>>> H_1(X^o) @>>> H_1(X) @>>> 0\\
\end{CD}
$$
By the snake lemma we obtain an exact sequence
$$0 \rightarrow I_X/I_Y \rightarrow H_1(M) \rightarrow H_1(Y) \rightarrow H_1(X) \rightarrow 0.$$
Now $H_1(M) \simeq \bZ/n\bZ$ and $H_1(Y)$ is finite of order coprime to $n$ by assumption.
Hence $I_X/I_Y=H_1(M)$ and $H_1(Y)=H_1(X)$.
Now consider the $p$-part of the exact sequence
$$H_1(L) \rightarrow H_1(X^o) \oplus H_1(M) \rightarrow H_1(Y) \rightarrow 0$$
for $p$ a prime.
If $p$ divides $n$, then since $H_1(L) \simeq \bZ/n^2\bZ$, $H_1(M) \simeq \bZ/n\bZ$, and $H_1(Y)_{(p)}=0$ we find that $H_1(X^o)_{(p)}=0$.
If $p$ does not divide $n$, then $H_1(X^o)_{(p)} = H_1(Y)_{(p)}$. Thus $H_1(X^o)=H_1(Y)$. Combining, $H_1(X^o)=H_1(X)$.
Now the Mayer--Vietoris sequence for $X=X^o \cup B$ shows that $H_2(X) \rightarrow H_1(L)$ is surjective as required.

The statement \ref{mainthmintro}(1) holds by Lemma~\ref{MV} and the surjectivity of (\ref{restriction}) proved above.
The stability statement is given by Proposition~\ref{stability}.

It remains to establish the stated properties of the Chern classes of $F$.
We have $c_1(E)=nc_1(A) \in H_2(X,\bZ)$ because $E^{\vee\vee}=A^{\oplus n}$ and $E$ is torsion-free, and $c_1(F)=c_1(E) \in H_2(Y,\bZ) \subset H_2(X,\bZ)$ by flatness of $\cE$. The formula for $c_2(F)$ holds because $F$ is exceptional, see Remark~\ref{RR}.
Since $c_1(F)=nc_1(A)$ and $K_Y=K_X$ in $H_2(Y,\bZ) \subset H_2(X,\bZ)$ we can compute $c_1(F) \cdot K_Y$ modulo $n$ by a local computation at the singular point $P \in X$. Identify $(P \in X)$ with  $(0 \in \bC^2_{u,v}/\frac{1}{n^2}(1,na-1))$. By construction $c_1(A)$ is locally represented by the class of the Weil divisor $D:=(v=0)$ (because the strict transform $D'$ of this divisor in $\tilde{X}_1$ is Cartier and satisfies $D' \cdot C = 1$). The canonical divisor $K_X$ is locally represented by $-(uv=0) \sim -na(u=0)$. Thus the local intersection number $(nc_1(A) \cdot K_X)_P$ is given by
$$(nc_1(A) \cdot K_X)_P = n(v=0)\cdot(-na(u=0)) = -n^2a/n^2 = -a \bmod n.$$
It follows that $c_1(A)=c_1(F)/n$ generates $H_2(X,\bZ)/H_2(Y,\bZ) \simeq \bZ/n\bZ$ because $(a,n)=1$.
\end{proof}

\section{Example: The projective plane}\label{P2}

We analyze our construction in the case $Y=\bP^2$. We use the classification of exceptional bundles on $Y$ \cite{DLP85},\cite{R89} and the classification of degenerations $Y \leadsto X$ \cite{HP10} to establish a bijective correspondence, see Theorem~\ref{P2thm}.

\begin{thm}\label{manetti}\cite[1.2]{HP10}
Let $X$ be a normal surface with quotient singularities which admits a smoothing to the projective plane.
Then $X$ is one of the following:
\begin{enumerate}
\item A weighted projective plane $\bP(a^2,b^2,c^2)$, where $(a,b,c)$ is a solution of the Markov equation
$$a^2+b^2+c^2=3abc.$$
\item A $\bQ$-Gorenstein deformation of one of the toric surfaces in (1), determined by specifying a subset of the singularities to be smoothed.
\end{enumerate}
\end{thm}

The solutions of the Markov equation are easily described: $(1,1,1)$ is a solution, and all solutions are obtained from $(1,1,1)$ by a sequence of \emph{mutations} of the form
\begin{equation}\label{mutation}
(a,b,c) \mapsto (a,b,c'=3ab-c).
\end{equation}
We can define a graph $G$ with vertices labelled by solutions of the Markov equation and edges corresponding to pairs of solutions related by a single mutation. Then $G$ is an infinite tree such that every vertex has degree $3$. See \cite[II.3]{C57}.

The surface $\bP=\bP(a^2,b^2,c^2)$ has cyclic quotient singularities of types $\frac{1}{a^2}(b^2,c^2)$, $\frac{1}{b^2}(c^2,a^2)$, $\frac{1}{c^2}(a^2,b^2)$.
Using the Markov equation one sees that these are Wahl singularities (note that $a,b,c$ are coprime and not divisible by $3$ by the inductive description of the solutions of the Markov equation above).
Moreover there are no locally trivial deformations and no local-to-global obstructions because $H^1(T_{\bP})=H^2(T_{\bP})=0$.
Thus the versal $\bQ$-Gorenstein deformation space of $\bP$ maps isomorphically to the product of the versal $\bQ$-Gorenstein deformation spaces of its singularities, which are smooth of dimension $1$ (see \S2).

\begin{prop}\label{localunique}
Let $X$ be a normal surface with quotient singularities which admits a smoothing to the projective plane.
Then $X$ is uniquely determined by its singularities.
\end{prop}
\begin{proof}
If $X$ is smooth then necessarily $X$ is isomorphic to $\bP^2$.
Now suppose $X$ has $r$ singularities with indices $a_1,\ldots,a_r$.
By Theorem~\ref{manetti} we have $r \le 3$ and the surface $X$ is obtained from a weighted projective plane $\bP=\bP(a_1^2,a_2^2,a_3^2)$ by smoothing the singularity of index $a_i$ for each $i>r$,
where $(a_1,a_2,a_3)$ is a solution of the Markov equation.
If $r=3$ then $X=\bP$ and clearly $X$ is determined by its singularities.
If $r=2$ then there are exactly two possibilities for $a_3$, related by the mutation $a'_3=3a_1a_2-a_3$.
By Example~\ref{connect} below these two choices yield isomorphic surfaces.

Finally suppose $r=1$.
Let $\frac{1}{n^2}(1,na-1)$ be the isomorphism type of the singularity $P \in X$.
Thus $n=a_1$ and $\frac{1}{n^2}(1,na-1) \simeq  \frac{1}{a_1^2}(a_2^2,a_3^2)$. Equivalently, using the Markov equation,
\begin{equation} \label{Markovslope}
\pm a= ((a_2^2+a_3^2)/a_1) \cdot (a_2^2)^{-1} = (3a_2a_3-a_1) \cdot (a_2^2)^{-1} = 3a_2^{-1}a_3 \bmod n
\end{equation}
(the sign ambiguity comes from interchanging the orbifold coordinates).
By inductively replacing $(a_1,a_2,a_3)$ by a mutation at $a_2$ or $a_3$ (and appealing to Example~\ref{connect} again), we may assume that $a_1= \max(a_1,a_2,a_3)$, cf. \cite[p.~27]{C57}.
Now by \cite[3.2]{R89} $(a_1,a_2,a_3)$ is uniquely determined by $n$ and $\pm a \bmod n$.
\end{proof}

\begin{ex}\label{connect}
Here we describe a two parameter family of surfaces which ``connects'' the weighted projective planes $\bP:=\bP(a^2,b^2,c^2)$, $\bP':=\bP(a^2,b^2,{c'}^2)$ associated to two solutions of the Markov equation related by a single mutation.
The family is given by
$$\cX=(XY=sZ^{c'}+tT^c) \subset \bP(a^2,b^2,c,c') \times \bC^2_{s,t}.$$
The special fibre $X:=\cX_0$ is the union of two weighted projective planes $\bP(a^2,c,c')$, $\bP(b^2,c,c')$ glued along the coordinate lines of degree $a^2$ and $b^2$. It has two Wahl singularities of indices $a$ and $b$ and two orbifold normal crossing singularities of indices $c$ and $c'$.
The fibers $\cX_{s,t}$ for $s = 0$, $t \neq 0$ are isomorphic to $\bP=\bP(a^2,b^2,c^2)$, with the embedding being the $c$-uple embedding
$$\bP(a^2,b^2,c^2) \rightarrow (XY=T^c) \subset \bP(a^2,b^2,c,c')$$
$$(U , V , W) \mapsto (X , Y , Z, T) = (U^{c},V^{c},W,UV).$$
Similarly, the fibers $\cX_{s,t}$ for $s \neq 0$, $t=0$ are isomorphic to $\bP'=\bP(a^2,b^2,{c'}^2)$.
The fibers $\cX_{s,t}$ for $s \neq 0$, $t \neq 0$ are obtained from $\bP$ or $\bP'$ by smoothing the singularity of index $c$ or $c'$ respectively.
\end{ex}

\begin{thm}\label{P2thm}
Let $S$ denote the set of isomorphism classes of normal surfaces $X$ such that $X$ has a unique singular point $P \in X$ which is a quotient singularity and $X$ admits a smoothing to $\bP^2$. (Then $P \in X$ is a Wahl singularity and the smoothing is necessarily $\bQ$-Gorenstein.)

Let $T$ denote the set of isomorphism classes of exceptional vector bundles $F$ on $\bP^2$ of rank greater than $1$ modulo the operations $F \mapsto F^{\vee}$
and $F \mapsto F \otimes L$ for $L$ a line bundle on $\bP^2$.

Then Theorem~\ref{mainthmintro} defines a bijection of sets
$$\Phi \colon S \rightarrow T, \quad [X] \mapsto [F].$$
\end{thm}
\begin{proof}
Let $X$ be a surface as in the statement. The singularity $P \in X$ is a Wahl singularity by Theorem~\ref{manetti}.
Let $P \in X$ be of type $\frac{1}{n^2}(1,na-1)$.

The smoothing of $X$ to $\bP^2$ is automatically $\bQ$-Gorenstein by \cite{M91}, \S1, Corollary~5.
Let $F$ denote an associated exceptional bundle $F$ on $Y=\bP^2$ given by Theorem~\ref{mainthmintro}.
Then $\rk(F)=n$ and $c_1(F) \cdot K_Y = \pm a \bmod n$.
Let $H$ denote the hyperplane class on $\bP^2$.
Then
\begin{equation}\label{Markovslope2}
3(c_1(F) \cdot H) = \pm a \bmod n,
\end{equation} 
and the slope $\mu(F):=(c_1(F) \cdot H)/\rk(F) \in \bQ$ is uniquely determined modulo translation by $\bZ$ and multiplication by $\pm 1$.
An exceptional vector bundle on $\bP^2$ is uniquely determined by its slope \cite[4.3]{DLP85}.
It follows that $F$ is uniquely determined up to $F \mapsto F \otimes L$ and $F \mapsto F^{\vee}$.
Thus the map $\Phi$ is well defined.

By Proposition~\ref{localunique} the surface $X$ is uniquely determined by the isomorphism type of its singularity, which is given by $n$ and $\pm a \mod n$.
This data is determined by $[F] \in T$ as above, so $\Phi$ is injective.
If $F$ is an exceptional vector bundle on $\bP^2$, then there exists a Markov triple $(a_1,a_2,a_3)$ such that
$\rk(F)=a_1$ and $(c_1(F)\cdot H)=\pm a_2^{-1}a_3 \bmod a_1$ \cite[3.2]{R89}.
Let $X$ be the surface obtained from $\bP(a_1^2,a_2^2,a_3^2)$ by smoothing the singularities of index $a_2$ and $a_3$.
Then by (\ref{Markovslope2}) and (\ref{Markovslope}) we have $[F]=\Phi([X])$. So $\Phi$ is surjective.
\end{proof}

\begin{rem}
An sequence of exceptional bundles $F_1,\ldots,F_N$ on a surface $Y$ is called an \emph{exceptional collection} if for all $j>k$ we have $\Ext^i(F_j,F_k)=0$ for each $i$. Under certain hypotheses, given a degeneration $Y \leadsto X$ of a smooth surface $Y$ to a normal surface $X$ with several Wahl singularities, we can construct an exceptional bundle $F_i$ on $Y$ associated to each singularity $P_i \in X$ such that the $F_i$ form an exceptional collection.
(Roughly speaking, we require $H^{2,0}(Y)=H^1(Y)=0$, (\ref{classgplocalglobal}) holds, and the singularities $P_i \in X$ are connected by a nodal chain of smooth rational curves.)
In the case $Y=\bP^2$ this yields a bijective correspondence between degenerations $X$ with several Wahl singularities and exceptional collections modulo a natural equivalence relation. We will give more details in a sequel to this paper.
\end{rem}

\begin{rem}
If $Y$ is a del Pezzo surface, we can show the following weaker result: every exceptional bundle $F$ on $Y$ is obtained by the construction of Theorem~\ref{mainthmintro}. The proof uses the classification of exceptional bundles on del Pezzo surfaces \cite{KO95}.
\end{rem}

\section{Background}\label{background}

\subsection{Reflexive sheaves}\label{reflexive}
Let $X$ be a normal variety. For $E$ a coherent sheaf on $X$ we write $E^{\vee}:=\cHom(E,\cO_X)$ for the dual of $E$.
We say $E$ is \emph{reflexive} if the natural map $E \rightarrow E^{\vee\vee}$ is an isomorphism.
Equivalently, $E$ is reflexive if $E$ is torsion-free and for any inclusion $i \colon U \subset X$ of an open subset with complement of codimension at least $2$ we have $i_*(E|_U)=E$.
For a coherent sheaf $E$ we call $E^{\vee\vee}$ the \emph{reflexive hull} of $E$.

Let $X$ be a normal projective variety of dimension $d$. The first Chern class defines a map
$$c_1 \colon \Pic X \rightarrow H^2(X,\bZ).$$
It is an isomorphism if $H^1(\cO_X)=H^2(\cO_X)=0$ (by the exponential sequence).
Let $\Cl X$ denote the \emph{class group} of reflexive rank $1$ sheaves on $X$ modulo isomorphism.
Equivalently, $\Cl X$ is the group of Weil divisors on $X$ modulo linear equivalence. We write $\cO_X(D)$ for the sheaf associated to a Weil divisor $D$. Then $\Pic X \subset \Cl X$ and we have the map
$$c_1 \colon \Cl X \rightarrow H_{2d-2}(X,\bZ), \quad \cO_X(D) \mapsto [D].$$
which is compatible with $c_1 \colon \Pic X \rightarrow H^2(X,\bZ)$ via
$$H^2(X,\bZ) \rightarrow H_{2d-2}(X,\bZ), \quad \alpha \mapsto [X] \cap \alpha.$$

Finally suppose $X$ is a normal projective surface with quotient singularities and $H^1(\cO_X)=H^2(\cO_X)=0$.
Let $X^o \subset X$ denote the smooth locus.
Then
$$\Pic X = H^2(X,\bZ) \subset \Cl X = H^2(X^o,\bZ)=H_2(X,\bZ),$$
see \cite{K05}, Proposition~38.

\subsection{Toric geometry}
We use various constructions from toric geometry which we review briefly here. We refer to \cite{F93} for the basic definitions of toric geometry.
\subsubsection{Weighted projective space}\label{wps}
The \emph{weighted projective space} $\bP(w_0,\ldots,w_r)$ is defined as the quotient
$$\bP(w_0,\ldots,w_r):= (\bC^{r+1} \setminus \{0\})/\bC^{\times}$$
where the action is given by
$$\bC^{\times} \ni \lambda \colon (x_0,\ldots,x_r) \mapsto (\lambda^{w_0}x_0,\ldots,\lambda^{w_r}x_r).$$
Then $\bP=\bP(w_0,\ldots,w_r)$ is a normal projective variety of dimension $r$.
We may assume that any $r$ of the $w_i$ have no common factors.
We have weighted homogeneous coordinates $X_0,\ldots,X_r$. The variety $\bP$ is covered by affine patches
$$(X_i \neq 0) = \bC^r / \textstyle{\frac{1}{w_i}}(w_0,\ldots,\hat{w_i},\ldots,w_r).$$

The variety $\bP$ is the toric variety associated to the free abelian group $N=\bZ^{r+1}/\bZ\cdot(w_0,\ldots,w_r)$ and the fan $\Sigma$ in $N \otimes \bR$
of cones generated by proper subsets of the standard basis of $\bZ^{r+1}$.

The variety $\bP$ carries a rank $1$ reflexive sheaf $\cO_{\bP}(1)$ such that the global sections of its $n$th reflexive tensor power
$\cO_{\bP}(n):=(\cO_{\bP}(1)^{\otimes n})^{\vee\vee}$ are the weighted homogeneous polynomials of degree $n$.
The class group $\Cl(\bP)$ of rank $1$ reflexive sheaves modulo isomorphism is isomorphic to $\bZ$, generated by $\cO_{\bP}(1)$.
The canonical sheaf $\omega_{\bP}$ (the pushforward of the canonical line bundle from the smooth locus) is isomorphic to $\cO_{\bP}(-\sum w_i)$.
The first Chern class $c_1 \colon \Cl(\bP) \rightarrow H_{2r-2}(\bP,\bZ)$ is an isomorphism.
Write $A:=c_1(\cO_{\bP}(1))$ for the positive generator of $H_{2r-2}(\bP,\bZ)$. Then the intersection product is given by $A^r=1/(w_0 \cdots w_r)$.
\subsubsection{Weighted blowups}\label{wtdblowup}
Consider the cyclic quotient
$$X= \bC^d_{x_1,\ldots,x_d} /\textstyle{\frac{1}{r}}(a_1,\ldots,a_d).$$
The variety $X$ is the affine toric variety associated to the free abelian group
$$N=\bZ^d+\bZ\cdot\textstyle{\frac{1}{r}}(a_1,\ldots,a_d)$$
and the cone $\sigma \subset N \otimes \bR$ generated by the standard basis of $\bZ^d$.
(That is, writing $M=\Hom(N,\bZ)$ and $\sigma^* \subset M \otimes \bR$ for the dual cone, the semigroup ring
$\bC[\sigma^* \cap M]$ is the invariant ring for the action of $\bZ/r\bZ$ on $\bC[x_1,\ldots,x_d]$.)

Let $w=\frac{1}{r}(w_1,\ldots,w_d) \in N$ be a primitive vector contained in the interior of the cone $\sigma$.
Let $\tilde{\Sigma}$ be the fan with support $\sigma$ obtained by adding the ray $\bR_{\ge 0}\cdot w$ and subdividing $\sigma$ into the cones spanned by $w$ and the codimension $1$ faces of $\sigma$. Then the fan $\tilde{\Sigma}$ determines a proper birational toric morphism
$$\pi \colon \tilde{X} \rightarrow X$$
called the \emph{weighted blowup} of $P \in X$ with weights $\frac{1}{r}(w_1,\ldots,w_d)$ with respect to the orbifold coordinates $x_1,\ldots,x_d$.
The morphism $\pi$ restricts to an isomorphism over $X \setminus \{P\}$, and the exceptional locus $E=\pi^{-1}(P)$ is a quotient of the weighted projective space $\bP(w_1,\ldots,w_d)$ by the action of a finite abelian group.
The toric variety $\tilde{X}$ is covered by affine charts $U_1,\ldots,U_d$ corresponding to the maximal cones of $\Sigma$.

Assume for simplicity that $w$ generates $N/\bZ^d$. Then $E=\bP(w_1,\ldots,w_d)$, and the restriction of $\pi$ to the chart $U_1$ is given by
$$ \bC^d_{u,x'_2,\ldots,x'_d} / \textstyle{\frac{1}{w_1}}(-r,w_2,\ldots,w_d) 
\rightarrow  \bC^d_{x_1,\ldots,x_d} / \textstyle{\frac{1}{r}}(a_1,\ldots,a_d),$$
$$(u,x_2',\ldots,x_d') \mapsto (x_1,\ldots,x_d)=(u^{w_1/r},u^{w_2/r}x_2',\ldots,u^{w_d/r}x_d').$$
The other charts are described similarly.

\end{document}